\documentclass[a4paper,12pt]{amsart}

\usepackage{amssymb}
\usepackage[english]{babel}
\usepackage[alphabetic]{amsrefs}

\usepackage{etoolbox}
\apptocmd{\sloppy}{\hbadness 10000\relax}{}{}

\newtheorem{theorem}{Theorem}[section]
\newtheorem{lemma}[theorem]{Lemma}
\newtheorem{proposition}[theorem]{Proposition}

\theoremstyle{definition}
\newtheorem{definition}[theorem]{Definition}

\theoremstyle{remark}
\newtheorem{remark}[theorem]{Remark}

\newcommand{\NN}{\mathbb{N}}
\newcommand{\CC}{\mathbb{C}}

\newcommand{\hol}{\mathcal{O}}

\newcommand{\Aut}{\operatorname{\mathrm{Aut}}}
\newcommand{\AutA}{\operatorname{{\mathrm{Aut}}_{\mathrm{alg}}}}
\newcommand{\AAutH}{\operatorname{{\mathrm{AAut}}_{\mathrm{hol}}}}

\title[The density property for Gizatullin surfaces]{The density property for Gizatullin surfaces of type $[[0,0,-r_2,-r_3]]$}

\author[R.~Andrist]{Rafael~B.~Andrist}
\address{Rafael B. Andrist \\ Fakult\"at f\"ur Mathematik und Natur\-wissen\-schaften \\ Bergische Universit\"at Wuppertal \\ Germany}
\email{rafael.andrist@math.uni-wuppertal.de}

\author[F.~Kutzschebauch]{Frank~Kutzschebauch}
\address{Frank Kutzschebauch \\ Universit\"{a}t Bern \\ Mathematisches Institut \\Sidlerstrasse 5 \\ CH-3012 Bern \\ Switzerland}
\email{frank.kutzschebauch@math.unibe.ch}

\author[P.-M.~Poloni]{Pierre-Marie~Poloni}
\address{Pierre-Marie Poloni \\ Universit\"{a}t Bern \\ Mathematisches Institut \\Sidlerstrasse 5 \\ CH-3012 Bern \\ Switzerland}
\email{pierre.poloni@math.unibe.ch}

\thanks{The first author would like to thank the University of Bern for their hospitality. This paper grew out of a three month stay in Bern.}


\begin{document}

\begin{abstract}
Gizatullin surfaces of type $[[0,0,-r_2,-r_3]]$ can be described by the equations $yu=xP(x)$, $xv=uQ(u)$ and $yv=P(x)Q(u)$ in $\CC^4_{x,y,u,v}$ where $P$ and $Q$ are non-constant polynomials. We establish the algebraic density property for smooth Gizatullin surfaces of this type. Moreover we also prove the density property for smooth surfaces given by these equations when $P$ and $Q$ are holomorphic functions. 
\end{abstract}

\maketitle

\section{Introduction}

Motivated by the question how to characterize affine $n$-space among affine algebraic/Stein  manifolds and the natural attempt to use the automorphism group for this characterization, in the last decades affine algebraic varieties and Stein manifolds with big (infinite-dimensional) automorphism groups have been studied intensively.
Among others the following two notions expressing the fact that the automorphisms group of a manifold is big, have been proposed:
The (algebraic) density property and holomorphic flexibility with the former implying the latter. 
Both the density property and holomorphic flexibility imply that the manifold in question is an Oka--Forstneri\v{c} manifold, i.e., a suitable target for maps enjoying the homotopy principle, we refer the interested reader to the overview article \cite{K}. In its simplest form this homotopy principle says that continuous maps from Stein spaces to Oka--Forstneri\v{c} manifolds are homotopic to holomorphic maps.

This important notion of Oka--Forstneri\v{c} manifold has also recently emerged from the intensive studies
around the homotopy principle which goes back to the 1930s and has had an enormous impact
on the development of Complex Analysis with a constantly growing number of applications (for definitions and more information we refer the reader to the text book \cite{For}).
Let us just recall the 

\begin{definition}[\cite{Varolin1}]
A complex manifold $X$ has the \emph{density property} if the Lie algebra $\mathrm{Lie}_{\mathrm{hol}}(X)$ generated by $\CC$-complete holomorphic vector fields on $X$ is dense in the Lie algebra of all holomorphic vector fields $\mathrm{VF}_{\mathrm{hol}}(X)$ on $X$ w.r.t.\ compact-open topology.
\end{definition}

\begin{definition}[\cite{Varolin1}]
An algebraic manifold $X$ has the \emph{algebraic density property} if the Lie algebra $\mathrm{Lie}_{\mathrm{alg}}(X)$ generated by $\CC$-complete algebraic vector fields on $X$ coincides with the Lie algebra of all algebraic vector fields $\mathrm{VF}_{\mathrm{alg}}(X)$ on $X$.
\end{definition}

Surprisingly  affine $n$-space  is by far not the only manifold with this properties. A large number of examples of such highly symmetric objects has been found and their classification and the exact relations between all mentioned properties remain unclear even in complex dimension $2$.
In particular, we do not even know the description of Stein surfaces $X$ on which the group of holomorphic automorphisms $\mathrm{Aut}_\mathrm{hol}(X)$ acts transitively.

In the algebraic case the question of transitivity
was ``almost'' resolved in the papers of Gizatullin and Danilov \cite{Gi}, \cite{dan-giz-autos} and we need the following definition to state their result.

\begin{definition} We call a normal Stein (resp.\ affine algebraic) surface $X$ quasi-homogeneous with respect to a subgroup $G$ of the group of its holomorphic (resp.\ algebraic)
automorphisms if the natural action of $G$ has an open orbit in $X$ whose complement is at most finite.
A normal algebraic surface is called quasi-homogeneous (without any reference to a group) if it is quasi-homogeneous with respect to the group $\AutA (X)$ of algebraic automorphisms\footnote{We cite the footnote in \cite{KKL}*{p.\ 2}: When the complement to the open orbit is empty, one has transitivity.
However there are examples of smooth quasi-homogeneous surfaces for which the complements of the open orbits are not empty. In the case of surfaces over algebraically closed field of positive characteristic they appeared already in the paper of
Gizatullin and Danilov \cite{dan-giz-quasi} who also knew, but did not publish such examples for characteristic zero.
In a published form examples of complex quasi-homogeneous surfaces with non-empty complements can be found in a recent paper of Kovalenko \cite{nontrans}.}.
\end{definition}

With the exception of the two-dimensional torus $\CC^\ast \times \CC^\ast$  and $\CC \times \CC^\ast$, every normal open quasi-homogeneous surface admits a completion
$\bar X$ by a simple normal crossing divisor such that the dual graph $\Gamma$ of its boundary $\bar X \setminus X$ is a linear rational graph \cites{Gi, dan-giz-autos} which can be always chosen in the following standard form (a so-called \emph{standard zigzag})
\[
[[0, 0,-r_2 , \dots, -r_d]]
\]
where $d \geq 2$ and $r_j \geq 2$ for $j = 2, \dots, d$.

A surface admitting such a completion will be called a \emph{Gizatullin surface}.

Recall that a holomorphic vector field $\nu$ on a complex space $X$ is called complete if the solution of the ODE
\[
\frac{d}{dt} \varphi (x,t) = \nu (\varphi (x,t)), \quad \varphi (x,0) = x
\]
is defined for all complex times $t \in \CC$ and all initial values $x \in X$. The induced maps
$\Phi_t \colon X \to X$ given by $\Phi_t (x) = \varphi (x,t)$
yield the phase flow of $\nu$ which is nothing but a one-parameter subgroup in the group $\mathrm{Aut}_\mathrm{hol}$  of holomorphic
automorphism with parameter $t \in \CC_+$ (so-called holomorphic $\CC_+$-action). 
It is worth mentioning that when $X$ is an affine algebraic variety and $\nu$ is a complete algebraic vector field its phase flow may well be non-algebraic. 

Kaliman, Leuenberger and the second author recently extended this  quasi-homogeneity results  a certain step towards the analytic situation. Replacing locally nilpotent vector fields by complete algebraic vector fields and $\mathrm{Aut}_\mathrm{alg} (X)$ by its analogue, the group $\AAutH (X)$ of algebraically generated holomorphic automorphisms, they classified normal affine algebraic surfaces which are generalized Gizatullin, in terms of their boundary divisors in the completion (for the list see \cite{KKL}*{Theorem 1.4}).

\begin{definition} A holomorphic automorphism $\alpha$ of an algebraic variety $X$ will be called \emph{algebraically generated} if $\alpha$ coincides with an element
$\Phi_t$ of the phase flow of a complete algebraic vector field on $X$ as before.
The subgroup of the holomorphic automorphism group generated by such algebraically generated  automorphisms will be denoted by $\AAutH (X)$, 
and a normal affine algebraic surface $X$ quasi-homogeneous w.r.t.\ $\AAutH (X)$ will be called \emph{generalized Gizatullin surface}. 
\end{definition}

Thus studying the relation between flexibility properties, say for surfaces, the following question is of central importance:
\medskip

 \emph{Which smooth generalized Gitzatullin surfaces have the (algebraic) density property?}
\medskip

Concerning this question, only the following was known so far:
For so-called Danilov--Gizatullin surfaces, which are described by a zigzag $[[0, 0, -2, \dots, -2]]$ and completely determined up to isomorphism by the length of the zigzag, the density property was proved in \cite{dan-giz-dens}.

For $d = 2$, resp.\ a zigzag  $[[0, 0, -r_2]]$, one obtains the so-called Danielewski surfaces. Their algebraic automorphism group is well understood, and their holomorphic automorphism group was studied in \cites{Danielewski1, Danielewski2}. The density property for Danielewski surfaces is proved in \cite{KKhyper}.

\bigskip

In this paper we study the case $d = 3$. The algebraic automorphism group for these Gizatullin surfaces  has been investigated by Blanc and Dubouloz in \cite{alghuge}. They show that the algebraic automorphism group is large in the sense that any countable union of algebraic subgroups is strictly smaller than the automorphism group. While this phenomenon was known to occur for affine-algebraic manifolds of higher dimension, it is the first class of examples for algebraic surfaces. We use the following description of the above mentioned class proved in  their paper \cite{alghuge}*{Proposition 3.2}:

Any Gizatullin surface with a completion by a zigzag $[[0,0,-r_2,-r_3]]$ is isomorphic to an algebraic surface of the form $S_{P, Q}$ for non-constant polynomials $P$ and $Q$ defined as follows:

\begin{definition}\label{defi:surface}
We define the complex-analytic surface $S_{P,Q} \subset \CC^4_{x,y,u,v}$ by
\begin{equation}
\label{eqsurface}
\left\{ \begin{array}{lcl}
yu & = & x P(x) \\
xv & = & u Q(u) \\
yv & = & P(x) Q(u)
\end{array} \right.
\end{equation}
\end{definition}

\begin{remark}
The surfaces $S_{P,Q}$ are smooth if and only if $P$ and $Q$ have only simple zeros and not both of them vanish in $0$. For constant $P$ or $Q$ these equations define a Danielewski surface or $\CC \times \CC^\ast$.
\end{remark}

The main results of the paper resolve the above question for the Gizatullin surfaces with a completion by a zigzag $[[0,0,-r_2,-r_3]]$ and their holomorphic analogues, namely we prove:

\begin{theorem}
\label{mainsmoothalg}
The smooth Gizatullin surfaces $S_{P, Q}$ with polynomials $P$ and $Q$ enjoy the algebraic density property.
\end{theorem}

\begin{theorem}
\label{mainsmoothhol}
The smooth surfaces $S_{P, Q}$ with holomorphic functions $P$ and $Q$ enjoy the density property.
\end{theorem}

The paper is organized as follows. In Section \ref{secdensity} we recall the methods for proving the density property and the algebraic density property which were developed by Kaliman and the second author: The idea is to find a submodule in the Lie algebra generated by complete vector fields together with some transitivity conditions.

In Section \ref{seccoordinates} we prepare for some explicit calculations on the complex surfaces $S_{P, Q}$. In Section \ref{sectrans} we provide the first ingredient, the transitivity and in Section \ref{secsubmodule} the second ingredient, the existence of a submodule in the Lie algebra generated by complete vector fields. Finally, we deduce our Theorems in Section \ref{secmainproof}.

We thank Matthias Leuenberger for carefully reading the manuscript and checking all calculations, especially for pointing out a calculation error in an earlier version of the paper.

\section{Density Property}
\label{secdensity}

In this section we briefly recall the notions and Theorems we need for proving our result.

\begin{definition}[\cite{denscrit}*{Definition 2.2}] \hfill
\begin{enumerate}
\item Let $X$ be an algebraic manifold and $x_0 \in X$.
A finite subset $M$ of the tangent space $T_{x_0} X$ is called a \emph{generating set}
if the image of $M$ under the action of the isotropy subgroup of $x_0$
(in the group of all algebraic automorphisms $\mathrm{Aut}_\mathrm{alg}(X)$)
generates the whole space $T_{x_0} X$.
\item Let $X$ be a complex manifold and $x_0 \in X$.
A finite subset $M$ of the tangent space $T_{x_0} X$ is called a \emph{generating set}
if the image of $M$ under the action of the isotropy subgroup of $x_0$
(in the group of all holomorphic automorphisms $\mathrm{Aut}_\mathrm{hol}(X)$)
generates the whole space $T_{x_0} X$.
\end{enumerate}

\end{definition}

We will make use of the following, central result of \cite{denscrit}:
\begin{theorem}\cite{denscrit}*{Theorem 1}
\label{modulealg}
Let $X$ be an affine algebraic manifold, homogeneous w.r.t.\ $\mathop{\mathrm{Aut}_{\mathrm{alg}}} X$, with algebra of regular functions $\CC[X]$, and $L$ be a submodule of the $\CC[X]$-module of all algebraic vector fields such that $L \subseteq \mathrm{Lie}_{\mathrm{alg}}(X)$.
Suppose that the fiber of $L$ over some $x_0 \in X$ contains a generating set.
Then $X$ has the algebraic density property.
\end{theorem}

The analogous statement in the holomorphic case follows with essentially the same proof:

\begin{theorem}
\label{modulehol}
Let $X$ be a Stein manifold, homogeneous w.r.t.\ $\mathop{\mathrm{Aut}_{\mathrm{hol}}}\! X$\!,
with algebra of holomorphic functions $\hol(X)$, and $L$ be a submodule of the $\hol(X)$-module of all holomorphic vector fields such that $L \subseteq \overline{\mathrm{Lie}_{\mathrm{hol}}(X)}$.
Suppose that the fiber of $L$ over some $x_0 \in X$ contains a generating set.
Then $X$ has the density property.
\end{theorem}

\begin{remark}
\label{overshearcomplete}
We remark that given a complete vector field $\Theta$ and a holomorphic function $f$ on a complex manifold $X$, the vector field $f \cdot \Theta$ is complete as well if and only if $\Theta(\Theta(f)) = 0$, see \cite{shears} and also \cite{fibred}*{Lemma 3.3} for an explicit formula.
\end{remark}

\section{Coordinates and complete fields}
\label{seccoordinates}

We shall often make use of the following three coordinate neighbourhoods, each of which is dense in the surface. For our calculations it will be convenient to know the push-forward of a vector field.

\begin{align*}
\CC \times \CC^\ast \ni &(x, y) \mathop{\mapsto}^\varphi \left( x, y, \frac{x P(x)}{y}, \frac{P(x)}{y} \cdot Q\left( \frac{x P(x)}{y} \right) \right) \in S_{P,Q} \\
&\frac{\partial}{\partial x} \mapsto \frac{\partial}{\partial x} + \left( \frac{x P'(x)}{y} + \frac{P(x)}{y} \right) \frac{\partial}{\partial u} \\ & \qquad\quad + \left( \frac{P'(x) Q(u)}{y} + \frac{u P'(x) Q'(u)}{y} + \frac{P^2(x)Q'(u)}{y^2} \right) \frac{\partial}{\partial v} \\
&\frac{\partial}{\partial y} \mapsto \frac{\partial}{\partial y} - \frac{u}{y} \frac{\partial}{\partial u}
- \left(\frac{v}{y} + \frac{u P(x) Q'(u)}{y^2}\right) \frac{\partial}{\partial v}
\end{align*}

\bigskip

\begin{align*}
\CC \times \CC^\ast \ni &(u, v) \mathop{\mapsto}^\psi \left( \frac{u Q(u)}{v}, \frac{Q(u)}{v} \cdot P\left( \frac{u Q(u)}{v} \right), u, v \right) \in S_{P,Q} \\
&\frac{\partial}{\partial u} \mapsto \frac{\partial}{\partial u} + \left( \frac{u Q'(u)}{v} + \frac{Q(u)}{v} \right) \frac{\partial}{\partial x} \\ & \qquad\quad + \left( \frac{Q'(u) P(x)}{v} + \frac{u Q'(u) P'(x)}{v} + \frac{Q^2(u)P'(x)}{v^2} \right) \frac{\partial}{\partial y} \\
&\frac{\partial}{\partial v} \mapsto \frac{\partial}{\partial v} - \frac{x}{v} \frac{\partial}{\partial x}
- \left(\frac{y}{v} + \frac{x Q(u) P'(x)}{v^2}\right) \frac{\partial}{\partial y}
\end{align*}

\bigskip

\begin{align*}
\CC^\ast \times \CC^\ast \ni (x, u) &\mathop{\mapsto}^\chi \left(x, \frac{x P(x)}{u}, u, \frac{u Q(u)}{x}\right) \in S_{P,Q} \\
\frac{\partial}{\partial x} &\mapsto \frac{\partial}{\partial x}
 + \frac{P(x) + x P'(x)}{u} \frac{\partial}{\partial y}
 - \frac{v}{x} \frac{\partial}{\partial v} \\
\frac{\partial}{\partial u} &\mapsto \frac{\partial}{\partial u}
 - \frac{y}{u} \frac{\partial}{\partial y}
 + \frac{Q(u) + u Q'(u)}{x} \frac{\partial}{\partial v}
\end{align*}

\begin{lemma}
\label{easycomplete}
The following vector fields extend to vector fields on $S_{P,Q}$ and are $\CC$-complete:	
\[
\begin{split}
\varphi_\ast\left(y^2 \frac{\partial}{\partial x} \right), \;
\varphi_\ast\left( x y \frac{\partial}{\partial x} \right), \;
\varphi_\ast\left( x y \frac{\partial}{\partial y} \right), \\
\psi_\ast\left( v^2 \frac{\partial}{\partial u} \right), \;
\psi_\ast\left( u v \frac{\partial}{\partial u} \right), \;
\psi_\ast\left( u v \frac{\partial}{\partial v} \right), \\
\chi_\ast\left( x u \frac{\partial}{\partial x} \right), \;
\chi_\ast\left( x u \frac{\partial}{\partial u} \right)
\end{split}
\]
\end{lemma}
\begin{proof}
Using the explicit formulae given above and the defining equations \eqref{eqsurface} of $S_{P,Q}$ it is immediate that these vector fields extend polynomially resp.\ holomorphically to $S_{P, Q}$.

In the coordinates given by $\varphi$ it is straightforward that the flow maps of $y^2 \frac{\partial}{\partial x}$ resp.\ $x y \frac{\partial}{\partial x}$ resp.\ $x y \frac{\partial}{\partial y}$ are given by $(x, y, t) \mapsto (x + y^2 t, y)$ resp.\ $(x, y, t) \mapsto (\exp(y t) x, y)$ resp.\ $(x, y, t) \mapsto (x, \exp(x t) y)$ and exist for all $t \in \CC$. Since the vector fields extend from the open and dense set $\varphi(\CC \times \CC^\ast)$ to $S_{P,Q}$, it is easy to see that the flow maps extend as complete flow maps to $S_{P,Q}$. Similarly for the vector fields given in the coordinates $\psi$ and $\chi$.
\end{proof}

\begin{remark}If $P(0)=0$, then one checks that the  vector field $\varphi_\ast\left(y \frac{\partial}{\partial x} \right)$  extends to the following locally nilpotent derivation on $S_{P,Q}$:  
\begin{align*}
\varphi_\ast\left(y \frac{\partial}{\partial x} \right)=y\frac{\partial}{\partial x}&+\left(P(x)+xP'(x)\right)\frac{\partial}{\partial u}+\\&+\left(P'(x)Q(u)+P'(x)uQ'(u)+\frac{P(x)}{x}uQ'(u)\right)\frac{\partial}{\partial v}
\end{align*}

Similarly, if $Q(0)=0$ then $\psi_\ast\left(v \frac{\partial}{\partial u} \right)$ extends to a locally nilpotent derivation on $S_{P,Q}$.
\end{remark}

\section{Transitivity}
\label{sectrans}
In this section we show that the automorphism group acts transitively on $S_{P,Q}$ in the smooth case. We first consider the holomorphic and then the algebraic situation.

\begin{proposition}
\label{prop:holotrans}
Let $P$ and $Q$ be holomorphic functions with only simple zeros and assume further that either $P(0) \neq 0$ or $Q(0) \neq 0$. Then the subgroup of holomorphic automorphism generated by the following complete vector fields acts transitively on the complex surface $S_{P,Q}$.
\[
\begin{split}
\chi_\ast\left( x u \frac{\partial}{\partial x} \right), \;
\chi_\ast\left( x u \frac{\partial}{\partial u} \right),
\varphi_\ast\left(y^2 \frac{\partial}{\partial x} \right), \\
\varphi_\ast\left( x y \frac{\partial}{\partial x} \right), \;
\psi_\ast\left( v^2 \frac{\partial}{\partial u} \right), \;
\psi_\ast\left( u v \frac{\partial}{\partial u} \right), \;
\end{split}
\]
\end{proposition}
\begin{proof}
Without loss of generality we may assume that $P(0) \neq 0$ due the ``formal'' symmetry $(x,y,P) \leftrightarrow (u,v,Q)$.
The flow maps of $x u \frac{\partial}{\partial x}$ and $x u \frac{\partial}{\partial u}$ together act transitively on $\CC^\ast_x \times \CC^\ast_u$.
We only need to consider the following cases, where we need to move the points inside the coordinates given by $\chi$:
\begin{enumerate}
\item $x \neq 0, u = 0 \Longrightarrow v = 0, P(x) = 0, y \in \CC$

Here, the vector field $\varphi_\ast\left( y^2 \frac{\partial}{\partial x} \right)$ takes the following form:
\[
y^2 \frac{\partial}{\partial x} + x y P'(x) \frac{\partial}{\partial u} + y P'(x) Q(0) \frac{\partial}{\partial v}
\]
Except for $y = 0$ we can change the $u$-coordinate while keeping $x \neq 0$.
If $y = 0$, then we consider
\[
\varphi_\ast\left( xy \frac{\partial}{\partial x} \right) = x^2 P'(x) \frac{\partial}{\partial u} + x P'(x) Q(0) \frac{\partial}{\partial v}
\]
Since $P(x) = 0$, we have $P'(x) \neq 0$ and can change the $u$-coordinate while keeping $x \neq 0$.
\item $x = 0, u \neq 0$: similarly with $\psi_\ast\left( v^2 \frac{\partial}{\partial u} \right)$ for $v \neq 0$ and with  $\psi_\ast\left( u v \frac{\partial}{\partial u} \right)$ for $v = 0$.
\item $x=0, u=0$: The vector fields take the following form:
\begin{align*}
\qquad \varphi_\ast\left( y^2 \frac{\partial}{\partial x} \right) &= y^2 \frac{\partial}{\partial x} + y P(0) \frac{\partial}{\partial u} + \left( y P'(0) Q(0) + P^2(0) Q'(0) \right) \frac{\partial}{\partial v} \\
\qquad \varphi_\ast\left( v^2 \frac{\partial}{\partial u} \right) &= v^2 \frac{\partial}{\partial u} + v Q(0) \frac{\partial}{\partial x} + \left( v Q'(0) P(0) + Q^2(0) P'(0) \right) \frac{\partial}{\partial y}
\end{align*}
Hence we can change the $x$-coordinate by $\varphi_\ast\left( y^2 \frac{\partial}{\partial x} \right)$ except for $y = 0$ and the $u$-coordinate by $\varphi_\ast\left( v^2 \frac{\partial}{\partial u} \right)$ except for $v = 0$.
It remains to take care of $(0,0,0,0)$, where necessarily $Q(0) = 0$ and $Q'(0) \neq 0$. Therefore,
\begin{align*}
\left.\varphi_\ast\left( y^2 \frac{\partial}{\partial x} \right)\right|_{(0,0,0,0)} &= P^2(0) Q'(0) \frac{\partial}{\partial v}
\end{align*}
is non-vanishing and we are done. \qedhere
\end{enumerate}
\end{proof}

Next we deal with the algebraic situation.
Observe that the flow maps of the locally nilpotent derivations $\varphi_\ast\left(y^2 \frac{\partial}{\partial x} \right)$ and $\psi_\ast\left( v^2 \frac{\partial}{\partial u} \right)$, given in Lemma \ref{easycomplete}, act transitively on the open subset $\{(x,y,u,v)\in S_{P,Q}\mid y\neq 0\text{ or }v\neq 0\}$. It thus only remains to consider a finite number of points, namely the points of the  subset
\[\Lambda_{P,Q}=\{p_0=(x_0,0,u_0,0)\mid x_0P(x_0)=u_0Q(u_0)=0\}\subset S_{P,Q},\]
and the question is now to decide, whether there exists, given a point $p_0\in\Lambda_{P,Q}$, an algebraic  $\CC_+$-action on $S_{P,Q}$ whose flow map moves $p_0$ outside of $\Lambda_{P,Q}$. Note that this would be the case if we can find another surface $S_{\widetilde{P},\widetilde{Q}}$ and an isomorphism $\theta \colon S_{P,Q} \to S_{\widetilde{P},\widetilde{Q}}$ such that $\theta(p_0)\notin\Lambda_{\widetilde{P},\widetilde{Q}}$.

As observed in \cite{alghuge}, if $P(0)\neq0$, then the surface $S_{P,Q}$ is isomorphic to every surface $S_{P,\widetilde{Q}}$, where  $\widetilde{Q}$ is defined by $\widetilde{Q}(u)=Q(u-\lambda)$ for some constant $\lambda\in\CC$. In the sequel, we will use these isomorphisms, that we can actually give explicitly.

\begin{lemma}\label{lemma:explicitIso}
Let $P\in\CC[x]$, $Q\in\CC[u]$ and  $\lambda \in \CC$. Set
 \[\widetilde{Q}(u) = Q(u - \lambda P(0))\]
 and define an algebraic automorphism $\Theta$ of $\CC^4$ by
\[\begin{pmatrix}x\\y\\u\\v\end{pmatrix}\mapsto\begin{pmatrix}X:=x+\lambda y\\y\\
U:=u+\dfrac{XP(X)-xP(x)}{y}\\v+\dfrac{P(X)\widetilde{Q}(U)-P(x)\widetilde{Q}(\overline{U})}{y}+u\dfrac{\widetilde{Q}(\overline{U})-Q(u)}{x}\end{pmatrix},\]
where $\overline{U}$ denotes the polynomial $\overline{U}=u+\lambda(P(x)+xP'(x))$.
Then, $\Theta$ induces an isomorphism  $\theta:S_{P,Q}\to S_{P,\widetilde{Q}}$.
\end{lemma}

\begin{proof} The proof is a straight-forward calculation.
\end{proof}

\begin{proposition}
\label{prop:algtrans}
Let $P,Q$ be polynomials.
If both $P$ and $Q$ have only simple zeros, then the subgroup of $\Aut_{alg}(S_{P,Q})$ generated by all algebraic $\CC_+$-actions on $S_{P,Q}$ acts transitively on the smooth locus $(S_{P,Q})_{\textrm{reg}}$.
\end{proposition}

\begin{proof}
Observe first that we can suppose that $P(0)$ and $Q(0)$ are not both non-zero. Indeed, if $Q(0)\neq0$, then there exists a $\lambda\in\CC$ such that $\widetilde{P}(0)=0$, where $\widetilde{P}$ denotes the polynomial $\widetilde{P}(x)=P(x-\lambda Q(0))$. Since $S_{P,Q}$ and $S_{\widetilde{P},Q}$ are isomorphic by Lemma \ref{lemma:explicitIso}, we can also suppose that $P(0)=0$.

Let $p_0=(x_0,0,u_0,0)$ be a point in $\Lambda_{P,Q}$. If $P(x_0)+x_0P'(x_0)\neq0$, then we can choose a constant $\lambda\in\CC$ such that $u_0+\lambda(P(x_0)+x_0P'(x_0))\neq0$ and $Q(u_0+\lambda( P(x_0)+x_0P'(x_0)))\neq0$. Therefore, Lemma \ref{lemma:explicitIso} gives us an isomorphism $\theta \colon S_{P,Q}\to S_{P,\widetilde{Q}}=S_{P,Q}$ such that $\theta(p_0)\notin\Lambda_{P,Q}$, because $u_0+\lambda(P(x_0)+x_0P'(x_0))$ is nothing but the third coordinate of $\theta(p_0)$. Consequently, there exists an additive group action on $S_{P,Q}$ which moves $p_0$ outside the set $\Lambda_{P,Q}$.

It remains to consider the case where $P(x_0)+x_0P'(x_0)=0$. Actually, it only remains the case $x_0=0$. Indeed,  it follows from the first defining equation of $S_{P,Q}$,  that $x_0P(x_0)=0$. But, $x_0\neq0$ would imply that $P(x_0)=0$, thus that $x_0P'(x_0)=P(x_0)+x_0P'(x_0)=0$, and finally that $x_0$ is a multiple root of $P$.

So, let $p_0=(0,0,u_0,0)$. We apply again the automorphism $\theta:S_{P,Q}\to S_{P,\widetilde{Q}}=S_{P,Q}$ given in Lemma \ref{lemma:explicitIso}. We will conclude the proof by showing that the image $p_1=\theta(p_0)$ of $p_0$ is not in $\Lambda_{P,Q}$. For this, it suffices to show that  the last coordinate of $p_1$ is not equal to zero. One checks that the coordinates of $p_1$ are $(x_1,y_1,u_1,v_1)$ with $x_1=x_0=0$, $y_1=y_0=0$, $u_1=u_0+\lambda(P(x_0)+x_0P'(x_0))=u_0$ and
\[v_1=\lambda P'(0)(Q(u_0)+2u_0Q'(u_0)).\]
Since $P(0)=0$ and $P$ has no multiple roots, we have that $P'(0)\neq0$. On the other hand, $u_0$ satisfies the equality $u_0Q(u_0)=0$. Therefore, we have three possibilities: either $u_0\neq0$ and $Q(u_0)=0$ (and in this case $v_1\neq0$ follows because $Q$ has only simple roots),  $u_0=0$ and $Q(u_0)\neq0$ (and in this cas again, we have $v_1\neq0$), or $u_0=0$ and $Q(0)=Q(u_0)=0$ (and in this case, $p_0=(0,0,0,0)$ was a singular point of $S_{P,Q})$.
\end{proof}

\begin{remark}Note that Proposition \ref{prop:algtrans} cannot be generalized to all surfaces $S_{P,Q}$. For instance, consider the Gizatullin surface $S$ defined by the equations 
\begin{equation*}
\left\{ \begin{array}{lcl}
yu & = & x^2 (x-1)^2 \\
xv & = & u^2 (u-1)^2 \\
yv & = & x(x-1)^2 u(u-1)^2.
\end{array} \right.
\end{equation*}
Then, the point $p=(1,0,1,0)\in S$ is  smooth and is fixed by all algebraic automorphisms of $S$. Indeed \cite{alghuge}*{Proposition 5.1.} implies that $\Aut_{alg}(S)$ is generated by the involution $(x,y,u,v)\mapsto(u,v,x,y)$ and the subgroup $J_y$ of automorphisms of $S$ preserving the $\mathbb{A}^1$-fibration $\text{pr}_y \colon S \to \mathbb{A}^1$. The latter consisting of the flow maps of all locally nilpotent derivations  $\varphi_\ast\left(yR(y) \frac{\partial}{\partial x} \right)$ with $R(y)\in\CC[y]$.
\end{remark}

\section{Generating a submodule}
\label{secsubmodule}

\begin{proposition}
\label{anideal}
Both in the algebraic and the holomorphic case there exists a non-trivial ideal $I \subset \hol(S_{P,Q})$ such that $I \varphi_\ast \left( \frac{\partial}{\partial y} \right)$ is contained in the Lie algebra generated by complete vector fields on $\hol(S_{P,Q})$.
\end{proposition}

\begin{proof}[Proof of Proposition \ref{anideal}]
From Lemma \ref{easycomplete} and Remark \ref{overshearcomplete} we can conclude that for $j, k \in \NN_0$ the following vector fields are complete on $S_{P,Q}$:
\begin{equation}
 \varphi_\ast\left( y^{k+2} \frac{\partial}{\partial x} \right), \;
 \varphi_\ast\left( x y^{k+1} \frac{\partial}{\partial x} \right), \;
 \varphi_\ast\left(x^{j+1} y \frac{\partial}{\partial y} \right)
\end{equation}

In the coordinates $\varphi$ we obtain the following Lie combinations:
\begin{align*}
&\left[ y^{k+2} \frac{\partial}{\partial x}, x y \frac{\partial}{\partial y} \right] + (k+2) x y^{k+2} \frac{\partial}{\partial x} &= y^{k+3} \frac{\partial}{\partial y} \\
&(k+2)^{-1} \cdot \left[ y^{k+3} \frac{\partial}{\partial y}, x^j y \frac{\partial}{\partial y} \right] &=  x^j y^{k+3} \frac{\partial}{\partial y} \\
&(k+2)^{-1} \cdot \left[ x^j y \frac{\partial}{\partial y}, x y^{k+2} \frac{\partial}{\partial x} \right] + (k+2)^{-1} \cdot j x^j y^{k+3} \frac{\partial}{\partial y} &= x^{j+1} y^{k+2} \frac{\partial}{\partial x}
\end{align*}
In front of $\frac{\partial}{\partial y}$ we have obtained the ideal $(y^3)$ in the ring $\hol(\CC^2)$, but this set is not an ideal in the ring $\hol(S_{P,Q})$.
Again by Lemma \ref{easycomplete} and Remark \ref{overshearcomplete} we see that $\chi_\ast \left(x u^{\ell+1} \frac{\partial}{\partial x} \right)$ is complete for $\ell \in \NN_0$.

We continue our calculation in the coordinates $\varphi$. Observe that
\begin{align*}
\frac{\partial}{\partial x} u &= \frac{\partial}{\partial x} \frac{x P(x)}{y} & &= \frac{x P' (x) + P(x)}{y} \\
\frac{\partial}{\partial y} u &= \frac{\partial}{\partial y} \frac{x P(x)}{y} &= -\frac{x P(x)}{y^2} &= - \frac{u}{y}
\end{align*}

We evaluate the following Lie bracket:
\begin{equation}
\begin{split}
\label{uterms}
&\left[ \varphi^\ast \left( \chi_\ast \left(x u^{\ell+1} \frac{\partial}{\partial x} \right) \right), \; x^j y^{k+5} \frac{\partial}{\partial y} \right] = \\ 
&\left[ x u^{\ell + 1} \frac{\partial}{\partial x} + (P(x) + xP'(x)) x u^\ell \frac{\partial}{\partial y}, \; x^j y^{k+5} \frac{\partial}{\partial y}\right] = \\
& (\ell+1) x^{j+2} P(x) y^{k+3} u^{\ell} \frac{\partial}{\partial x} \\
& + \left( j x^j y^{k+5} u^{\ell+1} + (k + \ell + 5) x^{j+1} \left( P(x) + x P'(x) \right) y^{k+4} u^{\ell} \right) \frac{\partial}{\partial y}
\end{split}
\end{equation}
By choosing $\ell = 0$, $j \geq 1, k \geq 0$ and subtracting known terms from the Lie algebra generated by complete vector fields, we obtain
\[
x^j y^{k+5} u \frac{\partial}{\partial y}
\]
We now want to apply equation \eqref{uterms} inductively for $\ell \in \NN$.
Observe that:
\begin{equation}
\left[ x^j y^{k+3} u^{\ell+1} \frac{\partial}{\partial y} , \; y^2 \frac{\partial}{\partial x} \right] = 2x^{j} y^{k+4} u^{\ell+1} \frac{\partial}{\partial x} - j x^{j-1} y^{k+5} u^{\ell + 1} \frac{\partial}{\partial y} 
\end{equation}
Hence, whenever we have obtained the terms $x^j y^{k+5} u^\ell \frac{\partial}{\partial y}, j \geq 1, k \geq 2,$ for a given $\ell \in \NN$, we also obtain the terms $x^{j+1} y^{k+4} u^\ell \frac{\partial}{\partial x}$. 
Induction in $\ell$ now yields
\[
x^{2+j} y^{7+k} u^{1+\ell} \frac{\partial}{\partial x} \quad \text{and} \quad x^{2+j} y^{7+k} u^{1+\ell} \frac{\partial}{\partial y}
\]
for all $j, k, \ell \in \NN_0$. Note that this also works if $P$ is not a polynomial.

By the ``formal'' symmetry $(x,y,P) \leftrightarrow (u,v,Q)$ we now obtain also the following vector fields:
\[
\psi_\ast \left( u^{2+j} v^{7+k} x^{1+\ell} \frac{\partial}{\partial u} \right) \quad \text{and} \quad
\psi_\ast \left( u^{2+j} v^{7+k} x^{1+\ell} \frac{\partial}{\partial v} \right)
\]
for all $j, k, \ell \in \NN_0$.

We can conclude that the following vector field is in the Lie algebra generated by the complete vector fields and that
\begin{multline}
\Lambda :=
\left(
v \psi_\ast \Big( u^{2+j} v^{7+k} x^{1+\ell} \frac{\partial}{\partial u} \right) + \\
(u Q'(u) + Q(u)) \psi_\ast \left( u^{2+j} v^{7+k} x^{1+\ell} \frac{\partial}{\partial v} \right) \Big)
\end{multline}
contains no $\frac{\partial}{\partial x}$-part, hence
\[
\varphi^\ast(\Lambda) = x^{1+s} u^{2+r} v^{1+m} R(x, u) \frac{\partial}{\partial y}
\]
for some polynomial resp.\ holomorphic function $R$. We might rewrite $\tilde{R}(x, u) = x u^2 R(x, u)$. Finally, we choose $s = r = 0$ and evaluate the following Lie bracket:
\[
\left[ x^{2+j} y^{3+k} u^{1+\ell} \frac{\partial}{\partial y}, v^{m+1} \tilde{R}(u,x) \frac{\partial}{\partial y} \right]
= x^{1+j} y^{1+k} u^{1+\ell} v^m T(x,y,u,v) \frac{\partial}{\partial y}
\]
with some polynomial resp.\ holomorphic function $T$. Thus, we have obtained the ideal $I = (xyu T)$ in front of $\frac{\partial}{\partial y}$.
\end{proof}

\section{Proof ot the Main Theorems}
\label{secmainproof}
We are now in the situation to prove the Theorems \ref{mainsmoothalg} and \ref{mainsmoothhol}.

\begin{proof}[Proof of Theorem \ref{mainsmoothalg}]
By Proposition \ref{prop:algtrans} the smooth complex surface $S_{P, Q}$ is a $\AutA(S_{P,Q})$-homogeneous manifold. We want to apply Theorem \ref{modulealg}: The desired submodule $L$ is provided by Proposition \ref{anideal}. It only remains to show that we obtain a generating set: pick a point $p_0 = (x_0, y_0, u_0, v_0) \in  S_{P,Q}$ such that we find a vector field $\mu \in L \neq 0$ which does not vanish in $p_0$ and such that $y_0 \neq 0$. The locally nilpotent derivation $\nu := \varphi_\ast\left( y^2 \frac{\partial}{\partial x} \right)$ does not vanish in $p_0$ and the function $f(x,y,u,v) := y - y_0$ lies in its kernel. By Remark \ref{overshearcomplete}, $f \cdot \nu$ is again complete, and its flow map will fix the point $p_0$. The induced action of the time-$1$-map on the tangent space $T_{p_0} S_{P,Q}$ is given by $w \mapsto w + d_{p_0} f \nu(p_0) = w + {y_0}^2 \cdot \varphi_{\ast}\left( \frac{\partial}{\partial x} \right)$ (see e.g.\ the calculation in \cite{denscrit}*{Corollary 2.8}). Evaluating this expression for $w = \mu(p_0)$ and noting that $\mu(p_0)$ is a non-zero vector in direction of $\varphi_{\ast}\left( \frac{\partial}{\partial y} \right)$ we see that $\mu(p_0)$ is a generating set.
\end{proof}

\begin{proof}[Proof of Theorem \ref{mainsmoothhol}]
The proof is completely analogous to the algebraic situation. The homogeneity is provided by Proposition \ref{prop:holotrans} instead. Formally, $\nu$ will be defined in the same way, however it will typically not be a locally nilpotent derivation, but still a complete vector field.
\end{proof}

\emergencystretch=1em

\end{document}